\documentclass[11pt,twoside]{article}
\usepackage[english]{babel}

\usepackage{amsthm}
\usepackage{amsmath}
\usepackage{amsfonts}
\usepackage{amssymb}

\usepackage[arrow, matrix, curve]{xy}
\usepackage{picinpar, graphicx}

\usepackage{fancyhdr}
\def\shorttitle{\sc Geometry, Topology, QFT and Cosmology}
\leftmark
\def\shortauthor{\sc\leftmark}

\newtheorem{Thm}{Theorem}
\newtheorem{Def}{Definition}

\newtheorem{Lem}{Lemma}
\newtheorem{Prop}{Proposition}
\newtheorem{Cor}{Corollary}
\newtheorem{Q}{Question}
\newtheorem{Ex}{Example}

\begin{document}

\title{Energy-momentum's local conservation laws\\ and  generalized isometric embeddings of\\ vector bundles}

\author{Nabil Kahouadji}

\date{}





\maketitle

\thispagestyle{empty}

\begin{abstract}
\noindent Using exterior differential systems, we provide a positive answer to the generalized isometric embedding problem of vector bundles, and show how conservation laws for a class of PDE can be\linebreak constructed, for instance, for covariant divergence-free energy-\linebreak momentum tensors.
\end{abstract}

\section{Introduction}
Exterior differential systems play an important role in mathematics in general as well as in differential geometry  since they are a way of studying PDE from a geometric viewpoint. After the emergence of abstract notion of manifolds \cite{Gauss-Art} and \cite{Riemann-Art}, which are defined, roughly speaking, as  Hausdorff topological spaces that possess  local embeddings in\linebreak Euclidean spaces (local charts), the question of the existence of an isometric embedding of Riemannian manifolds into Euclidean spaces emerged naturally. Another way to express this question is to find whether any abstract Riemannian manifold is in fact a submanifold of a Euclidean space.  In 1871, Schlaefli \cite{Schlaefli-Art} discussed the local form of this embedding and conjectured that a neighborhood in an $m$-dimensional Riemannian manifold would generally require an embedding space of $m(m+1)/2$ dimensions. Janet \cite{Janet-Art} and Cartan \cite{Cartan-Art} gave a positive answer in the local and analytic case. A global existence of the isometric embedding of Riemannian manifolds in the smooth case was established by Nash \cite{Nash-Art,NashIE-Art} and Kuiper \cite{Kuiper-Art} in the smooth case. A new problem is considered  in \cite{LCLGIEVB-Art} where the author addresses the question of the existence of isometric embedding of vector bundles, in a sense that will be later explained, as a natural generalization of the classical isometric embedding problem and used to construct conservation laws for a class of PDE. For that purpose, the following paper is organized as follows:  in section 2 we recall the basic definitions and results (proofs and details may be found in  \cite{M2Thesis}) concerning connections on vector bundles  that allow us to state in section 3 the problem of constructing conservation laws by generalized isometric embeddings of vector bundles, originally stated by F.~H\'elein (see \cite{Helein-Book}). We also give some examples and motivations of  the problem. In section 4 we give a brief introduction to exterior differential systems and  Cartan-K\"{a}hler theory. In section 5 we state the author's main theorem in \cite{LCLGIEVB-Art} and explain the principal steps of the proof. Finally, in section 6 we present an application for energy-momentum tensors.

\section{On Cartan's Structure Equations}
Let $\xi=(\mathbb{V},\pi,\mathcal{M})$ be a vector bundle over a smooth $m$-dimensional manifold $\mathcal{M}$ with an $r$-dimensional vector space $\mathbb{V}$ as a standard fiber. Denote by $(\Gamma(\mathrm{T}\mathcal{M}), [, ])$ the Lie algebra of vector fields on $\mathcal{M}$ and $\Gamma(\mathbb{V})$ the moduli space of cross-sections of the vector bundle $\mathbb{V}$.

\subsection{Connection on a vector bundles}  
A connection on a vector bundle is a way of ``differentiating'' cross-sections along vector fields, and it is usually defined as a bilinear operator  $\nabla$  on $\Gamma(\mathrm{T}\mathcal{M})\times\Gamma(\mathbb{V})$ with values on $\Gamma(\mathbb{V})$ which satisfy $\nabla_{(fX)}S=f\nabla_{X}S$ and a Leibniz identity type  $\nabla_{X}(fS)=X(f)S+f\nabla_{X}S$ for all $X \in \Gamma(\mathrm{T}\mathcal{M})$ and for all $ S \in \Gamma(\mathbb{V})$.\\

Given a connection $\nabla$, we define its curvature as 
a trilinear operator $\mathcal{R}^{\nabla}$  on $\Gamma(\mathrm{T}\mathcal{M})\times\Gamma(\mathrm{T}\mathcal{M})\times\Gamma(\mathbb{V})$ with values on  $\Gamma(\mathbb{V})$ which associates any cross-section $S$ and any two vector fields $X$ and $Y$ with the cross section
 $$\mathcal{R}^{\nabla}(X,Y)S=\Big([\nabla_{X},\nabla_{Y}]-\nabla_{[X,Y]}\Big)S$$

From the definition, one can easily check the following property:

\begin{Thm}\label{ThCourbureConnexion}
For any $f,g$ and $h$ smooth functions on $\mathcal{M}$, $S\in \Gamma(E)$ a section of $\xi$ and
$X,Y \in \Gamma(\mathrm{T}\mathcal{M})$ two tangent vector fields of  $\mathcal{M}$, we have
\begin{equation}
\mathcal{R}^{\nabla}(fX,gY)(hS)=f.g.h.\mathcal{R}^{\nabla}(X,Y)S
\end{equation}
\end{Thm}

Depending on the situation,  expressing the connection and its curvature following Cartan's formalism seems to be more convenient: let us  denote by $\mathcal{O}$ an open set of $\mathcal{M}$. A set of $r$ local sections $S=(S_{1},S_{2},\dots , S_{r})$ of $\xi$ is called a moving frame  (or a frame field) if  for all $p $ in $
\mathcal{O}$, $S(p)=\Big(S_{1}(p),S_{2}(p),\dots , S_{r}(p)\Big)$ form a basis of the fiber  $\mathbb{V}_{p}$ over the point $p$.

If we consider $X\in \Gamma(\mathrm{T}\mathcal{M})$ a tangent vector field on  $\mathcal{M}$, then  since $\nabla_{X}S_{j}$ is another section of  $\xi$, it can be expressed in the moving frame $S$ as follows:
\begin{equation}
\nabla_{X}S_{j}=\sum_{i=1}^{r}\omega^{i}_{j}(X)S_{i}
\end{equation}
where $\omega^{i}_{j}\in \Gamma(\mathrm{T}^\ast\mathcal{M})$ are differential 1-forms on $\mathcal{M}$. 
\begin{Def}
The $r\times r$ matrix $\omega=(\omega^{i}_{j})$ whose entries are differential 1-forms is called the  connection 1-form of $\nabla$.
\end{Def}

The connection $\nabla$ is completely determined by the matrix
$\omega=(\omega^{i}_{j})$. Conversely, a matrix of differential 1-forms on $\mathcal{M}$ determines a connection (in a non-invariant way depending on the choice of the moving frame). \\

Let $X,Y \in \Gamma(\mathrm{T}\mathcal{M})$ be two tangent vector fields. As previously, since $\mathcal{R}^\nabla(X,Y)S_{j}$ is a section of $\xi$, it  can be expressed on the moving frame $S$ as follows:
\begin{equation}
\mathcal{R}^\nabla(X,Y)S_{j}=\sum_{i=1}^{r}\Omega^{i}_{j}(X,Y)S_{i}
\end{equation}
where $\Omega^{i}_{j}\in \Gamma(\wedge^{2}\mathrm{T}^\ast\mathcal{M})$ are differential 2-forms on $\mathcal{M}$. 

\begin{Def}
The $r\times r$ matrix $\Omega=(\Omega^{i}_{j})$ whose entries are differential 2-forms, is called the  curvature 2-form of the connection $\nabla$.
\end{Def}

With this viewpoint, we can  state the following theorem that gives the relation between the connection 1-form $\omega$ and  the curvature 2-form $\Omega$.
\begin{Thm}\label{ThFormConnexFormCourbure}
\begin{equation}\label{EquaStructCartan}
\mathrm{d}\omega+\omega\wedge\omega = \Omega \quad\text{(matrix form)}
\end{equation}
or
\begin{equation}\label{2ndequastructureCartan}
\mathrm{d}\omega^{i}_{j}+\sum_{k=1}^{r}\omega^{i}_{k}\wedge\omega^{k}_{j}=\Omega^{i}_{j} \quad\text{(on components)}
\end{equation}
\end{Thm}

This viewpoint is convenient for endowing connection on a vector bundle when we have a connection on another vector bundle. Actually, since connections and their curvatures are expressed by  differential forms, one can use the usual operations on differential forms, for instance, the pull-back, to induce these objects on another vector bundle.  Let us consider $\xi=(\mathbb{V},\pi, \mathcal{M})$ and $\xi
'=(\mathbb{V}',\pi ', \mathcal{M})$ to  be two vector bundles on $\mathcal{M}$ of the same rank.
Consider a map  $f:\mathcal{M}\longrightarrow \mathcal{M}$ and denote 
$\tilde{f}:\mathbb{V}\longrightarrow \mathbb{V}'$ the associated bundle map i.e. $(f, \tilde{f})$ satisfies the following commutative diagram:
$$
\begin{xy}
\xymatrix{
      \mathbb{V} \ar[r]^{\tilde{f}} \ar[d]_{\pi}    &   \mathbb{V}' \ar[d]^{\pi '}\\
      \mathcal{M} \ar[r]_{f}             &   \mathcal{M}}
\end{xy}
$$

If $\nabla '$ is a connection on $\mathbb{V}'$, the vector bundle morphism induces a pull-back connection on $\mathbb{V}$ by 
$ \nabla=\tilde{f}^{\ast}\nabla ' $  such that for any $S'\in \Gamma(\mathbb{V}')$ and $X\in \Gamma(\mathrm{T}\mathcal{M})$, $\nabla_{X}(\tilde{f}^{\ast}S')=(\tilde{f}^{\ast}\nabla')_{X}(\tilde{f}^{\ast}S')=\tilde{f}^{\ast}\Big(\nabla_{f_{\ast}X}S'\Big)$ where $f_{\ast ,p}:\mathrm{T}_{p}\mathcal{M}\longrightarrow \mathrm{T}_{f(p)}\mathcal{M}$ is the linear tangent map.

We can also induce a connection on $\xi$ in another way. 
The connection $\nabla '$ is completely determined by the matrix of differential 
 $1$-forms $\omega'=({\omega '}^{i}_{j})$. Consequently,  the connection $\nabla$ can be defined  by the matrix 
$\omega$ whose entries $\omega^{i}_{j}$ are the pull-back of ${\omega '}^{i}_{j}$ by $\tilde{f}$, i.e., $\omega=\tilde{f}^{\ast}\omega '$. Since the pull-back commutes with the exterior differentiation and with the wedge product\footnote{$\mathrm{d}(f^{\ast}\alpha)=f^{\ast}(\mathrm{d}\alpha)$ and
$f^{\ast}(\alpha\wedge\beta)=f^{\ast}(\alpha)\wedge
f^{\ast}(\beta)$ for all $\alpha , \beta \in \Gamma(\wedge \mathrm{T}^\ast\mathcal{M})$.}, the curvature 2-form  $\Omega$ of the connection $\nabla$  is the pull-back of the curvature 2-form of  $\nabla '$, i.e.
$\Omega=\tilde{f}^{\ast}\Omega '$. \\

An interesting property of a connection when the vector bundle is endowed with a Riemannian metric is to be ``compatible'' with that metric. Let us recall that a Riemannian metric $g$ on $\xi$ is  a positively-defined scalar product on each fiber. Then  
a connection $\nabla$ is said to be compatible with the metric  $g$, or a metric connection,  if  $\nabla$ satisfies  the Leibniz identity:
$\nabla_{X}\Big(g(S_{1},S_{2})\Big)=g(\nabla_{X}S_{1},S_{2})+g(S_{1},
\nabla_{X}S_{2})$ for all  $S_{1},S_{2}\in \Gamma(\mathbb{V})$  and
for all $ X\in \Gamma(\mathrm{T}\mathcal{M}).$ \\

The following result shows an interesting property of metric connections and will be useful in answering one case of  the generalized isometric embedding problem (stated in the next section).

\begin{Prop}\label{PropMatriceConnexAntiSym}
Let $S=(S_{1},S_{2},\dots , S_{n})$ be an orthonormal moving frame with respect to $g$, i.e. $g_{p}(S_{i},S_{j})=\delta_{ij}$ for all
$p\in \mathcal{O}$, $i,j=1,\dots ,r$, then the matrix of  1-forms $\omega$ associated to  $S$  and  the curvature matrix of 2-forms $\Omega$  are both skew-symmetric, i.e., $\omega^{i}_{j}+\omega^{j}_{i}=0$ and $\Omega^{i}_{j}+\Omega^{j}_{i}=0$.
\end{Prop}

This means that  metric connections and their curvatures are $\mathfrak{o}(n)$-valued differential forms  rather than just being $\mathfrak{gl}(n)$-valued differential forms.

\subsection{Tangent Bundle Case} We consider in this subsection a fundamental class of vector bundles which is the tangent bundle of a manifold, i.e., $\mathbb{V} = \mathrm{T}\mathcal{M}$.  All of the results stated above obviously remain true. However, this class of vector bundles provides more notions. For instance, let us consider, as previous, a local moving frame $S=(S_{1},
\dots , S_{m})$ over $\mathcal{O}\subset \mathcal{M}$. One can naturally associates  $S$ with a moving coframe $\eta=(\eta^{1},\dots , \eta^{m})$ defined as a local frame field of 1-forms, such that for all $p\in \mathcal{U}$, $\eta^{i}(p)(S_{j})=\delta^{i}_{j}.$\\

We define the torsion 2-form of $\nabla$ to be the $\mathrm{T}\mathcal{M}$-valued differential 2-form $\Theta=(\Theta^{i})$ such that $\Theta^{i}= \mathrm{d}\eta^{i}+\omega^{i}_{j}\wedge\eta^{j}$. This can be written in a more condensed manner and is said to be the first Cartan structure equation
\begin{equation}\label{1ereequaStructureCartan}
\mathrm{d}\eta+\omega\wedge\eta = \Theta
\end{equation}

The torsion of a connection measures the default for a connection to have  the parallelogram property. When $\Theta$ vanishes, the connection is said to be torsion-free. Moreover, a torsion-free connection which is also compatible with a Riemannian metric is said  to be a Levi-Civita connection.  Note that a Levi-Civita connection on a Riemannian manifold is unique.\\

The moving coframe $\eta$, the connection 1-form $\omega$, the curvature 2-form $\Omega$ and the torsion 2-form are connected as follows:

\begin{Prop}\label{PropBianchi} On a tangent bundle, the four forms
$\eta, \omega , \Theta$ and $\Omega$ are connected by the following systems of equations:
\begin{equation}
\mathrm{d}\Theta + \omega\wedge \Theta= \Omega\wedge \eta
\end{equation}
and
\begin{equation}\label{biachi1forme}
\mathrm{d}\Omega = \Omega\wedge \omega - \omega\wedge \Omega
\end{equation}
\end{Prop}

The equation (\ref{biachi1forme}) is the expression of the Bianchi identities via the connection  1-form and the curvature 2-form. Let us notice that equation (\ref{biachi1forme}) is also valid on an arbitrary vector bundle.

\subsection{Cartan's Structure Equations}

Let $(\mathcal{M},g)$ be  an $m$-dimensional Riemannian manifold. Let $\eta=(\eta^{1}\!,\eta^{2},$\linebreak $\dots , \eta^{m})$ be  an orthonormal moving coframe on $\mathcal{M}$. According to equations 
(\ref{1ereequaStructureCartan}), (\ref{2ndequastructureCartan}) and the proposition \ref{PropMatriceConnexAntiSym}, we establish the Cartan structure equations:

\begin{equation}
\begin{cases}
\displaystyle{\mathrm{d}\eta^{i}+\sum_{j=1}^{m}\omega^{i}_{j}\wedge\eta^{j}}=0\quad \text{ (torsion-free) }\\
\displaystyle{\mathrm{d}\omega^{i}_{j}+\sum_{k=1}^{m}\omega^{i}_{k}\wedge\omega^{k}_{j}=\Omega^{i}_{j}}
\end{cases}
\end{equation}
where the matrix $(\omega^{i}_{j})$ is the Levi-Civita connection 1-form (torsion-free connection which is compatible with the Riemannian metric $g$). Because $\eta$ is an orthonormal coframe field, $(\omega^{i}_{j})$ and $(\Omega^{i}_{j})$ are skew-symmetric.

\subsection{The Cartan Lemma} We end  this section with a technical lemma, which is  easy to establish and very useful for computations.  
\begin{Lem}\label{LemmedeCartan}
Let $\mathcal{M}$ be an  $m$-dimensional manifold.
$\omega^{1}, \omega^{2}, \dots, \omega^{r}$ a set of linearly independent differential  $1-$forms
 ($r\leq n$) and
$\theta^{1},\theta^{2},\dots , \theta^{r}$ differential  $1$-forms
such that
\begin{equation}
\sum_{i=1}^{r}\theta^{i}\wedge\omega^{i}=0
\end{equation}
then there exist $r^{2}$  functions $h^{i}_{j}$ in $\mathcal{C}^{1}(\mathcal{M})$ such that
\begin{equation}
\theta^{i}=\sum_{j=1}^{r}h^{i}_{j}\omega^{j}\quad \text{ with }
h^{i}_{j}=h^{j}_{i}.
\end{equation}
\end{Lem}

\section[A new problem: generalized isometric embeddings]{A new problem: generalized isometric\\ embeddings}

Many fundamental quantities in physics (mass, energy, movement quantity, momentum, electric charge, \dots), when some conditions are satisfied, do not change as the physical system evolves. One can then consider that there are \textit{conservation laws} that govern the evolution of a given physical system. From a mathematical viewpoint, a conservation law can be seen as a  map defined on a space $\mathcal{F}$ (which can be, for instance, a  function space, a fiber bundle section space, \dots) that associates each  element $f$ of $\mathcal{F}$ with a vector field $X$ on an $m$-dimensional Riemannian manifold $\mathcal{M}$, such that if $f$ is a solution of  a given PDE on $\mathcal{F}$, the vector field $X$ has a vanishing divergence. If we denote  by $g$ the Riemannian metric on the manifold  $\mathcal{M}$, we can canonically associate each vector field $X\in \Gamma(\mathrm{T}\mathcal{M})$ with a differential 1-form $\alpha_{X}: = g(X,\cdot)$. Since $\mathrm{div}(X)= \ast \mathrm{d} \ast \alpha_{X}$  (or $\mathrm{div}(X)\mathrm{vol}_{\mathcal{M}}= \mathrm{d}(X \lrcorner \, \mathrm{vol}_{\mathcal{M}}))$, 
 where $\ast$ is the Hodge operator, $\mathrm{vol}_{\mathcal{M}}$ is the volume form on $\mathcal{M}$, and $X\lrcorner \, \mathrm{vol}_{\mathcal{M}}$ is the interior product of $\mathrm{vol}_{\mathcal{M}}$ by the vector field $X$, the requirement $\mathrm{div}(X)=0$ may be replaced by the requirement  $\mathrm{d}(X \lrcorner \, \mathrm{vol}_{\mathcal{M}}) = 0$ and hence, conservation laws may also be seen   as  maps on $\mathcal{F}$ with values on  differential $(m-1)$-forms such that solutions of PDEs are mapped to closed differential $(m-1)$-forms on $\mathcal{M}$. More generally, we could extend the notion of conservation laws as mapping to differential $p$-forms (for instance, Maxwell equations in the vacuum can be expressed, as it is well-known, by requiring a system of differential 2-forms to be closed). As in \cite{LCLGIEVB-Art}, we address the question of finding conservation laws for a class of PDE described as follows:

\begin{Q} 
Let   $\mathbb{V}$  be  an $n$-dimensional vector bundle over $\mathcal{M}$. Let $g$ be a metric bundle  and $\nabla$ a connection that is compatible with that metric. We then have a covariant derivative $d_{\nabla}$ acting on vector valued differential forms. Assume that $\phi$ is a given  covariantly closed $\mathbb{V}$-valued differential $p$-form on $\mathcal{M}$, i.e.,
\begin{equation}\label{dnablaphi=0}
\mathrm{d}_{\nabla}\phi = 0.
\end{equation}
Does there exist $N\in \mathbb{N}$ and an embedding $\Psi$ of $\mathbb{V}$ into $\mathcal{M}\times \mathbb{R}^N$ given by $\Psi(x,X)=(x,\Psi_{x}X)$, where $\Psi_{x}$ is a linear map from $\mathbb{V}_{x}$ to $\mathbb{R}^N$ such that:
\begin{itemize}
\item $\Psi$ is isometric, i.e, for every $x\in \mathcal{M}$, the map $\Psi_{x}$ is an isometry,
\item If $\Psi(\phi)$ is the image of $\phi$ by $\Psi$, i.e., $\Psi(\phi)_{x} = \Psi_{x}\circ \phi_{x}$ for all $x\in \mathcal{M}$, then
\begin{equation}\label{dpsiphi=0}
\mathrm{d}\Psi(\phi)=0.
\end{equation}
\end{itemize}
\end{Q}

In this problem, the equation (\ref{dnablaphi=0}) represents  the given PDE (or a system of PDEs) and the map $\Psi$ plays the role of a conservation law. Note that the problem is trivial when the vector bundle is a line bundle. Indeed, the only connection on a real line bundle which is compatible with the metric is th flat one.

\subsection{Motivations:}
There are basically two main motivations to the statement of the above problem(see \cite{Helein-Book}) coming from geometry and physics:
\subsubsection{The isometric embedding problem:}

A fundamental example is the isometric embedding of Riemannian manifolds in Euclidean spaces. A map $u$ between two Riemannian manifolds $(\mathcal{M},g)$ and $(\mathcal{N},h)$ is said to be isometric if $u^\ast (h) =g$. After the emergence of the abstract notion of manifolds, due to the works of Gauss \cite{Gauss-Art} and Riemann\cite{Riemann-Art}, a natural question arose: does there exist an abstract manifold? Another way to express this question is to know if any given abstract manifold is in fact a submanifold of a certain Euclidean space or equivalently, does any arbitrary Riemannian manifold admit  an isometric embedding in a Euclidean space.  This problem is known as the isometric embedding problem which has been considered in various specializations and with assorted conditions. It is related to Question 1 as follows: $\mathcal{M}$ is an $m$-dimensional  Riemannian manifold, $\mathbb{V}$ is the tangent bundle $\mathrm{T}\mathcal{M}$, the connection $\nabla $ is the Levi-Civita connection, $p=1$  and the $\mathrm{T}\mathcal{M}$-valued  differential 1-form $\phi$ is the identity map on $\mathrm{T}\mathcal{M}$. Then  (\ref{dnablaphi=0}) expresses the torsion-free condition for the connection $\nabla$  and any solution $\Psi$ of (\ref{dpsiphi=0}) provides us with an isometric embedding $u$ of the\linebreak Riemannian manifold $\mathcal{M}$ into a Euclidean space $\mathbb{R}^N$ through $\mathrm{d}u=\mathrm{d}\Psi(\phi)$ and conversely. An answer to the local analytic isometric embeddings of Riemannian manifolds is given by the Cartan-Janet theorem:  

\begin{Thm}
Every $m$-dimensional real analytic Riemannian manifold can be
locally embedded isometrically in an
$\displaystyle{\frac{m(m+1)}{2}}$-dimensional\linebreak Euclidean space.
\end{Thm}

Janet \cite{Janet-Art} solved the local problem for an analytic Riemannian surface, and Cartan \cite{Cartan-Art} immediately extended the result to any $m$-manifold, treating it as an application to his theory of Pfaffian systems.  In \cite{NashIE-Art}, Nash solved the problem in the smooth and global case. Despite the fact that  the Cartan--Janet result is local and the analycity hypotheses on the  data may seem to be too restrictive, the Cartan--Janet theorem is important because it  actualizes the embedding in an optimal dimension,  unlike the Nash isometric embedding. 

Consequently, if the above problem has a positive answer for $p=1$, the notion of isometric embeddings of Riemannian manifolds is extended to the notion  of \textit{generalized isometric embeddings} of vector bundles. The general problem, when $p$ is arbitrary, can also be viewed as an \textit{embedding of covariantly closed vector valued differential $p$-forms}.

\subsubsection{Harmonic maps between Riemannian manifolds:}  A harmonic map $u$ between two Riemannian manifolds $(\mathcal{M},g)$ and $(\mathcal{N}, h)$ is a critical point of the Dirichlet functional
\begin{equation}
E[u]=\frac{1}{2}\int_{\mathcal{M}}|\mathrm{d}u|^2.
\end{equation}
Locally, the Euler-Lagrange systems  is expressed\footnote{We use here the usual convention on repeated indices.} by
\begin{equation}
\Delta_{g}u^{i} + g^{\alpha\beta}\Gamma^{i}_{jk}(u(x))\frac{\partial u^{j}}{\partial x^\alpha}\frac{\partial u^{j}}{\partial x^\alpha}=0 
\end{equation}
where $\Gamma^{i}_{jk}$ denote the Christoffel symbols of the connection on $\mathcal{N}$. Harmonic maps are in fact common to mathematicians and physicists. For instance,  when the target Riemannian manifold $(\mathcal{N},h)$ is replaced by $(\mathbb{R}, \langle ,\rangle_{\mathbb{R}})$, harmonic maps are harmonic functions on $(\mathcal{M},g)$. If the target manifold is $(\mathbb{R}^n, \langle,\rangle_{\mathbb{R}^n})$,  a map $u$ is harmonic if and only if each component of $u$ is a harmonic function on $\mathcal{M}$. Other examples of harmonic maps are isometries, geodesics and isometric immersions.

Harmonic maps are related to  Question 1  since they produce covariantly closed vector valued forms as expounded in \cite{Helein-Book}.  Indeed, harmonic maps can be characterized as follows:  let us consider a map $u$ defined on an $m$-dimensional Riemannian manifold $\mathcal{M}$ which takes values in an $n$-dimensional Riemannian manifold $\mathcal{N}$.  On the induced bundle by $u$ over $\mathcal{M}$, the $u^{\ast}\mathrm{T}\mathcal{N}$-valued differential  $(m-1)$-form $\ast \mathrm{d}u$ is covariantly closed if and only if the map $u$ is harmonic, where the connection on the induced bundle is the pull-back by $u$ of the Levi-Civita connection on $\mathcal{N}$.  A positive answer to Question 1 in this case will make it possible to construct conservation laws on $\mathcal{M}$ from covariantly closed  vector valued differential $(m-1)$-forms, provided, for example, by harmonic maps. In his book \cite{Helein-Book}, motivated by the question of the compactness of weakly harmonic maps in Sobolev spaces in the weak topology (which is still an open question), H\'elein considers harmonic maps between\linebreak Riemannian manifolds and explains how conservation laws may be obtained explicitly by   Noether's theorem if the target manifold is symmetric and formulates the problem for non-symmetric target manifolds.  
\\

\section{A short review of exterior differential systems}

In \cite{LCLGIEVB-Art}, the author gives a  positive answer  to the \textit{generalized isometric embedding} of an arbitrary vector bundle relative to a covariantly closed vector valued differential $(m-1)$-form in the local and analytic case, by reformulating the  problem in terms of an exterior differential system. Since the Cartan--K\"{a}hler theory plays an important role in the proof of Theorem \ref{CLGIEVBThm} and since the reader  may not be  familiar with exterior differential systems (EDS) and the Cartan--K\"{a}hler theorem, generalities are expounded concerning theses notions and results in this section. For details and proofs, the reader may consult \'Elie Cartan's book \cite{Cartan-Book} and the third chapter  of \cite{ExtDiffSys-Book}.

\subsection{EDS and exterior ideals}
Let us denote\footnote{This is a graded algebra under the wedge product. We do not use the standard notation $\Omega(\mathcal{M})$ so as to not confuse it with the curvature 2-form of the connection.} by $\Gamma (\wedge \mathrm{T}^{\ast}\mathcal{M})$ the space of smooth differential forms on $\mathcal{M}$. An exterior differential system is a finite set of differential forms $I\!=\!\{\omega_{1},\omega_{2},\dots , \omega_{k}\!\} \subset
\Gamma (\wedge \mathrm{T}^{\ast}\mathcal{M})$ for which we associate the set of equations $\{\omega_{i}=0 \,  | \, \omega_{i}\in I \}$.

\begin{Ex} $ I=\{ a\mathrm{d}x + b\mathrm{d}y +c\mathrm{d}z \}$ is an EDS on $\mathbb{R}^{3}$, where $a,b$ and $c$ are functions on $\mathbb{R}^3$. The EDS $I$ is said to be a Pfaffian system because it  contains only differential 1-form.
\end{Ex}

A subset of differential forms $\mathcal{I} \subset\Gamma (\wedge \mathrm{T}^{\ast}\mathcal{M})$  is an exterior ideal if the exterior product of any differential form of $\mathcal{I}$ by a differential form of $\Gamma (\wedge \mathrm{T}^{\ast}\mathcal{M})$ belongs to $\mathcal{I}$ and if the sum of any two differential forms of the same degree belonging to $\mathcal{I}$, belongs also to $\mathcal{I}$.
The exterior ideal generated by $I$ is  the smallest exterior ideal containing $I$. 

\begin{Ex}(continued)
Let us denote $\omega = a\mathrm{d}x + b\mathrm{d}y + c\mathrm{d}z$. The exterior ideal generated by $I$ is:
\begin{equation}
\mathcal{I} = \{ \alpha \wedge \omega | \alpha \in\Gamma (\wedge \mathrm{T}^{\ast}\mathbb{R}^3)\}
\end{equation}

\end{Ex}

An exterior ideal is said to be   an exterior differential ideal if it is closed under the exterior differentiation and hence,  the
exterior differential ideal generated by an EDS is the smallest
exterior differential ideal containing the EDS.
\begin{Ex}(continued) The exterior differential ideal generated by $I$ is:
\begin{equation}
\mathcal{I}_{1}=\{ \alpha\wedge \omega + \beta \wedge \mathrm{d}\omega | \alpha, \beta \in  \Gamma (\wedge \mathrm{T}^{\ast}\mathbb{R}^3)\} 
\end{equation}

\end{Ex}

 Let us notice that an EDS is closed if and only if the exterior differential ideal and the exterior ideal  generated by that EDS are  equal. In particular, if $I$ is an EDS,  $I\cup \mathrm{d}I$ is closed.\\

Let $I\subset\Gamma (\wedge \mathrm{T}^{\ast}\mathcal{M})$ be an EDS on $\mathcal{M}$ and let $\mathcal{N}$ be a submanifold of $\mathcal{M}$. The submanifold $\mathcal{N}$ is an integral manifold of  $I$ if  $\iota^{\ast}\varphi=0, \forall \varphi\in
I$, where $\iota$ is an embedding $\iota:\mathcal{N}\longrightarrow \mathcal{M}$. The purpose of this theory is to establish when a given EDS, which represents a PDE in a jet space, has or has not integral manifolds. We consider in this subsection, an $m$-dimensional real manifold $\mathcal{M}$ and $\mathcal{I}\subset \Gamma (\wedge \mathrm{T}^{\ast}\mathcal{M})$ an exterior differential ideal on $\mathcal{M}$.

An EDS which contains only differential $1$-forms is called a Pfaffian system. A necessary and sufficient condition to the existence of integral manifold for Pfaffian systems is given by the following theorem:

\begin{Thm}(Frobenius)\footnote{Note that the Frobenius  theorem may be stated using distribution of vector fields. Classical details may be found in \cite{M2Thesis}.} 
\\
Let $I=\{\omega^1 , \dots , \omega^r\}$ be a Pfaffian system on an $m$-dimensional manifold $\mathcal{M}$. Then a necessary and sufficient condition for the Pfaffian system to be completely integrable is:
\begin{equation}
\mathrm{d}\omega^{i}\wedge\omega^1\wedge\dots\wedge\omega^r = 0 \quad \text{for all}\quad i=1,\dots , r.
\end{equation}
\end{Thm}

\begin{Ex}(continued)
The necessary and sufficient condition for the existence of integral surfaces for $I$ is:
\begin{equation}
c\Big(\frac{\partial b}{\partial x} - \frac{\partial a}{\partial y }\Big) + b\Big(\frac{\partial c }{\partial x} - \frac{\partial a}{\partial z}\Big)+ a\Big(\frac{\partial c }{\partial y} - \frac{\partial b}{\partial z }\Big)=0.
\end{equation}
For instance, 
\begin{enumerate}
\item $I_{1}=\{ \mathrm{d}z - x\mathrm{d}y\}$ is not completely integrable in $\mathbb{R}^3$.
\item $I_{2}=\{ x\mathrm{d}x+y\mathrm{d}y+z\mathrm{d}z\}$ is completely integrable in $\mathbb{R}^3\smallsetminus\{ 0\}$.
\item $I_{3}=\{ z\mathrm{d}x+x\mathrm{d}y+y\mathrm{d}z\}$ is not completely integrable in $\mathbb{R}^3\smallsetminus\{ 0\}$.
\end{enumerate}
\end{Ex}

If an EDS contains differential 1-forms and functions, we can still apply the Frobenius theorem on the submanifold defined on the vanishing of these functions (outside the possible singularities). However, if the EDS contains differential forms of degree greater than 1, the Frobenius theorem is not helpful. For instance, let $\Omega$ be a closed differential 2-form on an $2m$-dimensional manifold $\mathcal{M}$  such that $\Omega^{m}\neq 0$ and $\Omega^{m+1} = 0$. The couple $(\mathcal{M}, \Omega)$ is called a \textit{symplectic manifold}. The integral $m$-dimensional manifolds of $\{ \Omega\}$, if they exist, are called Lagrangian manifolds. Finding then Lagrangian manifolds for a giving symplectic manifold is the same as searching the existence of integral manifolds of a differential 2-form. The following theory represents a powerful tool to answer this  question.\\

\subsection{Introduction to Cartan--K\"{a}hler theory}

\begin{Def} Let $M\in \mathcal{M}$. A linear subspace  $E$ of $ \mathrm{T}_{M}\mathcal{M}$ is an integral element of $\mathcal{I}$ if
$\varphi_{E}=0$ for all $\varphi \in \mathcal{I}$, where $\varphi_{E}$ means the evaluation of $\varphi$ on any basis of $E$. We denote by
$\mathcal{V}_{p}(\mathcal{I})$ the set of $p$-dimensional integral
elements of $\mathcal{I}$.
\end{Def}
$\mathcal{N}$ is an integral manifold of $\mathcal{I}$ if and only if each tangent space of  $\mathcal{N}$ is an integral element of $\mathcal{I}$. It is not hard to notice from the definition that a subspace of a given integral element is also an integral element.  We denote by $\mathcal{I}_{p}=\mathcal{I}\cap\Gamma (\wedge^p \mathrm{T}^{\ast}\mathcal{M})$ the set of differential $p$-forms of $\mathcal{I}$. Thus, 
$\mathcal{V}_{p}(\mathcal{I})=\{E\in G_{p}(\mathrm{T}\mathcal{M}) \, | \, \varphi_{E}=0$ for all $\varphi \in \mathcal{I}_{p} \}$.

\begin{Def}
Let $E$  be an integral element of $\mathcal{I}$. Let
$\{e_{1},e_{2},\dots ,e_{p}\}$ be a basis of  $E\subset \mathrm{T}_{M}\mathcal{M}$. The
polar space of $E$, denoted by $H(E)$, is the vector space defined
as follows:
\begin{equation}
H(E)=\{v\in \mathrm{T}_{M}\mathcal{M} \, |\,  \varphi(v, e_{1},e_{2},\dots , e_{p})=0
\text{ for all } \varphi \in \mathcal{I}_{p+1}\}.
\end{equation}
\end{Def}

Notice that $E\subset H(E)$.  The polar space plays an important role in exterior differential system theory as we can notice from  the following proposition.

\begin{Prop}
Let $E \in \mathcal{V}_{p}(\mathcal{I})$ be a $p$-dimensional
integral element of
 $\mathcal{I}$. A $(p+1)$-dimensional vector space $E^{+}\subset \mathrm{T}_{M}\mathcal{M}$ which contains $E$ is an integral element of  $\mathcal{I}$ if and only if $E^{+}\subset H(E)$.
\end{Prop}

In order to check if a given $p$-dimensional integral element of an
EDS $\mathcal{I}$ is contained in a
$(p+1)$-dimensional integral element of  $\mathcal{I}$ , we
introduce the following function called the extension rank $
r:\mathcal{V}_{p}(\mathcal{I}) \longrightarrow
\mathbb{Z}$ that associates each  integral element $E$ with an integer  
$r(E)\!\!\!=$\linebreak $\mathrm{dim} H(E)-(p+1)$. The extension rank  $r(E)$ is in fact the dimension of $\mathbb{P}(H(E)/E)$ and is always greater or equal than $-1$. If $r(E)=-1$, then $\mathrm{dim}H(E) = \mathrm{dim}E$ so that $H(E) = E$ and consequently, there is no hope of extending the integral element. An integral element E is said to be regular if $r(E)$ is constant on a  neighborhood of $M$.

\begin{Def}
An integral flag of  $\mathcal{I}$ in $M\in \mathcal{M}$ of length $n$ is a
sequence $(0)_{M}\subset E_{1}\subset E_{2}\subset \dots \subset
E_{n}\subset \mathrm{T}_{M}\mathcal{M}$ of integral elements $E_{k}$ of $\mathcal{I}$.
\end{Def}
An integral element $E$ is said to be ordinary
if its base point $M\in \mathcal{M}$ is an ordinary zero of $I_{0}=I\cap
\Gamma (\wedge^0 \mathrm{T}^{\ast}\mathcal{M})$ and if there exists an integral flag
$(0)_{M}\subset E_{1}\subset E_{2}\subset \dots \subset E_{n}=E
\subset \mathrm{T}_{z}\mathcal{M}$ where the $E_{k}$, $k=1,\dots ,(n-1)$ are
regular. Moreover, if  $E_{n}$ is itself regular, then  $E$ is said  to be regular. We can now state the following important results of the Cartan--K\"{a}hler theory.

\begin{Thm}\label{TestCartan}(Cartan's test)
Let  $\mathcal{I}\subset \Gamma (\wedge^\ast \mathrm{T}^{\ast}\mathcal{M})$ be  an exterior ideal which does not contain  0-forms (functions on $\mathcal{M}$). Let  $(0)_{M}\subset E_{1}\subset
E_{2}\subset \dots \subset E_{n}\subset \mathrm{T}_{M}\mathcal{M}$ be an integral flag of $\mathcal{I}$. For any $k<n$, we denote by $C_{k}$ the codimension of the polar space $H(E_{k})$ in $\mathrm{T}_{M}\mathcal{M}$. Then $\mathcal{V}_{n}(\mathcal{I})\subset G_{n}(\mathrm{T}\mathcal{M})$ is at least of codimension
$C_{0}+C_{1}+\dots + C_{n-1}$ at $E_{n}$.
Moreover, $E_{n}$ is an ordinary integral flag if and only if
$E_{n}$ has a neighborhood $\mathcal{O}$ in $G_{n}(\mathrm{T}\mathcal{M})$ such that
$\mathcal{V}_{n}(\mathcal{I})\cap \mathcal{O}$ is a manifold of codimension $C_{0}+C_{1}+\dots + C_{n-1}$  in  $\mathcal{O}$.
\end{Thm}

The numbers $C_{k}$ are called Cartan characters of the $k$-integral element. The following proposition is useful in the applications. It allows us to compute the Cartan characters of the constructed flag in the proof of the Theorem \ref{CLGIEVBThm}.

\begin{Prop}\label{PropCp}

At a point $M\in \mathcal{M}$,  let $E$ be an $n$-dimensional integral element of an exterior ideal $\mathcal{I}\cap\Gamma (\wedge^\ast \mathrm{T}^{\ast}\mathcal{M})$ which does not contain differential 0-forms.  Let $\omega_{1},\omega_{2},\dots ,
\omega_{n},\pi_{1},\pi_{2}, \dots , \pi_{s}$ (where $s=dim\, \mathcal{M} - n$) be a coframe in an open neighborhood of  $M\in M$ such that $E=\{v \in \mathrm{T}_{M}\mathcal{M}\,  | \, \pi_{a}(v)=0 \text{ for all } a=1,\dots ,  s\}$. For
all  $p\leqslant n$, we define $E_{p}=\{v\in E \, | \, \omega_{k}(v)=0 \text{ for all } k > p \}$. Let $\{\varphi_{1}, \varphi_{2},\dots
,\varphi_{r}\}$ be the set of differential forms which generate the
exterior ideal  $\mathcal{I}$, where $\varphi_{\rho}$ is of degree $(d_{\rho}+1)$.
For all $\rho$, there exists an expansion
\begin{equation}
\varphi_{\rho}=\sum_{|J|=d_{\rho}}\pi_{\rho}^{J}\wedge\omega_{J}+\tilde{\varphi}_{\rho}
\end{equation}
where the   1-forms $\pi_{\rho}^{J}$ are linear combinations of
the forms $\pi$ and the terms $\tilde{\varphi}_{\rho}$ are, either
of degree  2 or more on  $\pi$, or vanish at $z$.
Moreover, we have
\begin{equation}
H(E_{p})=\{v\in \mathrm{T}_{M}\mathcal{M} | \pi_{\rho}^{J}(v)=0 \text{ for all } \rho
\text{ and } \sup J\leqslant p\}
\end{equation}
In particular, for the integral flag $(0)_{z}\subset E_{1}\subset
E_{2}\subset\dots \subset E_{n}\cap \mathrm{T}_{z}\mathcal{M}$ of $\mathcal{I}$,
the Cartan characters $C_{p}$ correspond to the number of linear independent  forms
$\{\pi_{\rho}^{J}|_{z} \text{ such that } \sup J \leqslant p\}.$
\end{Prop}

The following theorem is a generalization in higher dimension of the well-known Cauchy problem.

\begin{Thm}[Cauchy-Kowalevskaya]
Let $y$ be a coordinate on $\mathbb{R}$, let $x=(x^{i})=$ be a coordinate on  $\mathbb{R}^n$, let $z=(z^a)$ be coordinate on $\mathbb{R}^s$, and let $p^{a}_{i}$ be coordinate on $\mathrm{R}^{ns}$.  Let $\mathcal{D} \subset \mathbb{R}^n\times\mathbb{R}\times\mathbb{R}^s\times\mathbb{R}^{ns}$ be an open domain, and let $G:\mathcal{D}\longrightarrow \mathbb{R}^s$ be a real analytic mapping. Let $\mathcal{D}_{0}\subset \mathbb{R}^n$ be an open domain and let $f:\mathcal{D}_{0}\longrightarrow \mathbb{R}^s$ be a real analytic mapping  so that the ''1-graph'' 
\begin{equation}
\Gamma_{f}=\{(x,y,f(x),\mathrm{D}f(x))| x\in \mathcal{D}_{0} \}
\end{equation}
lies in $\mathcal{D}$ for some constanta $y_{0}$, where $\mathrm{D}f(x)\in \mathbb{R}^{ns}$ is the Jacobian of $f$ described by the condition that $p^{a}_{i} (\mathrm{D}f(x)) = \partial f^{a}/\partial x^{i}$. Then, there exists an open neighborhood $\mathcal{D}_{1}\subset \mathcal{D}_{0}\times \mathbb{R}$ of $\mathcal{D}_{0}\times \{ y_{0}\}$ and a real analytic mapping $F:\mathcal{D}_{1}\longrightarrow \mathbb{R}^s $ which satsifies the PDE with initial condition
\begin{equation}\label{CKinitialData}
\begin{split}
\partial F/\partial y &= G(x,y,F,\partial F/\partial x)\\
F(x,y_{0}) &= f(x) \quad \text{ for all } x\in \mathcal{D}_{0}
\end{split}
\end{equation}
Moreover, $F$ is unique in the sense that any other real-analytic solution of  (\ref{CKinitialData}) agrees with $F$ in some neighborhood of $\mathcal{D}_{0}\times \{ y_{0}\}$.
\end{Thm}

The following theorem is of  great importance not only because it is a  generalization of the well-known Frobenius theorem but also represents a generalization of the  Cauchy--Kovalevskaya theorem.

\begin{Thm}\label{ThCartanKahler}(Cartan--K\"{a}hler)\\
Let $\mathcal{I}\subset\Gamma (\wedge^\ast \mathrm{T}^{\ast}\mathcal{M})$ be a real analytic exterior differential ideal which does not contain  functions. Let $\mathcal{X} \subset \mathcal{M}$ be a $p$-dimensional connected real analytic K\"{a}hler-regular integral manifold of $\mathcal{I}$.
Suppose that $r=r(\mathcal{X})\geqslant 0$. Let  $\mathcal{Z}\subset \mathcal{M}$ be a real analytic
submanifold of $\mathcal{M}$ of  codimension $r$ which contains  $\mathcal{X}$ and such
that  $T_{M}\mathcal{Z}$ and
$H(T_{M}\mathcal{X})$ are transversal in  $\mathrm{T}_{M}\mathcal{M}$ for all  $M\in \mathcal{X} \subset \mathcal{M}$.
There exists then a $(p+1)$-dimensional connected real
analytic integral manifold $\mathcal{Y}$ of $\mathcal{I}$, such that
$\mathcal{X}\subset \mathcal{Y} \subset \mathcal{Z}$. Moreover, $\mathcal{Y}$
is unique in the sense that another integral manifold of
$\mathcal{I}$ having the stated properties coincides with $\mathcal{Y}$ on an open neighborhood of  $\mathcal{X}$.
\end{Thm}

The analycity condition of the exterior differential ideal is crucial because of the requirements in the Cauchy--Kovalevskaya theorem used in the Cartan--K\"{a}hler theorem's proof. It has an important corollary. Actually, in the application, this corollary is often more used than the theorem and is sometimes called  \textit{the Cartan--K\"{a}hler theorem} in literature.

\begin{Cor}\label{CorCartanKahler}(Cartan--K\"{a}hler)\\
Let $\mathcal{I}$ be an analytic exterior differential ideal on a manifold
$\mathcal{M}$. If $E$ is an ordinary integral element of
$\mathcal{I}$, there exists an integral manifold of
$\mathcal{I}$ passing through $z$ and having $E$ as a tangent
space at  point $M$.
\end{Cor}

The following example, even ``trivial'',  illustrate how one can apply the Cartan-K\"{a}hler theory to non-Pfaffian systems.
\begin{Ex}
Let $(\mathbb{R}^4 , \mathrm{d}x^1\wedge\mathrm{d}x^3 + \mathrm{d}x^2\wedge\mathrm{d}x^4)$ be a symplectic space. Let us look for Lagrangian surfaces, i.e., integral surfaces for $\Omega$. Since $\Omega= \mathrm{d}x^1\wedge\mathrm{d}x^3 + \mathrm{d}x^2\wedge\mathrm{d}x^4$ is closed, the exterior ideal $ \mathcal{I}$ generated by $\Omega$ is closed.

$\mathcal{I}$ contains neither function nor differential 1-forms. Any point $M$ of $\mathbb{R}^4$ is an integral point and any tangent vector on $\mathrm{T}_{M}\mathbb{R}^4 \simeq \mathbb{R}^4$ is an integral 1-element. Let us notice that $r_{0} = 3$. Consequently, there exist integral 2-elements of $\mathcal{I}$. Let us consider  $\displaystyle{E_{1} = \text{span}\{\frac{\partial }{\partial x^1}\}}$ as an integral 1-element of $\mathcal{I}$.

Note that $H(E_{1}) = \{ X \in \mathrm{T}_{M}\mathbb{R}^4 | \Omega(X, E_{1}) = 0\}$. Hence $H(E_{1})$ is defined by the equation $X^1 = 0$, where $X^{i}$ are the components of the tangent vector $X$. Therefore, $C_{1} = 1$ and the extension rank $r_{1}= 2$, and there exist 2-integral elements of $\mathcal{I}$.  Consider $\displaystyle{E_{2} =span\{ \frac{\partial }{\partial x^1}, \frac{\partial }{\partial x^2}\}}$ is an integral 2-element of $\mathcal{I}$.  Since the extension ranks $r_{0}, r_{1}$ and $r_{2}$ are constant on a neighborhood  of $M$, the flag $M\subset E_{1} \subset E_{2}$ is regular. The Cartan-K\"{a}hler theorem assures then the existence of integral surfaces of $\mathcal{I}$.
 
 \end{Ex}

\section[The main result: construction of conservation laws]{The main result: construction of conservation laws}
In this section, we state Theorem \ref{CLGIEVBThm} concerning the existence of some generalized isometric embedding problem. We explain the main steps of the proof by setting a strategy for solving the problem in the general case (i.e., $m, n$ and $1\leqslant p \leqslant m-1$ are arbitrary).

\begin{Thm}[See \cite{LCLGIEVB-Art}]\label{CLGIEVBThm}
Let $\mathbb{V}$ be a real analytic  $n$-dimensional vector bundle  over a real analytic $m$-dimensional manifold $\mathcal{M}$ endowed with a metric $g$ and a connection $\nabla$ compatible with $g$. Given a non-vanishing covariantly closed $\mathbb{V}$-valued differential $(m-1)$-form $\phi$, there exists a local isometric embedding of $\mathbb{V}$ in  $\displaystyle{\mathcal{M}\times \mathbb{R}^{n+\kappa^n_{m,m-1}}}$ over $\mathcal{M}$ where $\displaystyle{\kappa^n_{m,m-1} \geqslant (m-1)(n-1)}$ such that the image of $\phi$ is a conservation law.
\end{Thm}



The existence Theorem \ref{CLGIEVBThm} can be applied for instance to harmonic maps. The last section of this paper is dedicated to presenting an application to energy-momentum tensors which occur e.g. in general relativity. 
\\

Our approach and strategy for solving Question 1 in the general case and for proving Theorem \ref{CLGIEVBThm} is the following: first, we reformulate the problem by means of an  exterior differential system on a manifold that have to be defined, and since all the data involved are real analytic, we use the Cartan--K\"{a}hler theory to prove the existence of integral manifolds, and hence to conclude the desired result.  The problem can be represented by the following diagram which summarizes the notations
\begin{equation}\nonumber
\begin{xy}
\xymatrix{ \mathbb{V}^{n}\ar[d]_{\mathrm{d}_{\nabla}\phi =0} \ar@^{(->}[r]^{\Psi\ \  \ \ \ }&  \mathcal{M}^{m}\times\mathbb{R}^{N}\ar[d]^{\mathrm{d}\Psi(\phi)=0}\\
 \mathcal{M}^{m} &\mathcal{M}^{m}}
\end{xy}
\end{equation}
where $N$ is an integer  that have to be defined in terms of the problem's  data: $n$, $m$ and $p$. Let us then set up a general strategy as an attempt to solve the general problem. We denote by $\kappa^n_{m,p}$ the  \textit{embedding codimension}, i.e., the dimension of the fiber extension in order to achieve the desired embedding. \\

 Let us first recast our problem by using moving frames and coframes. For convenience, we adopt the following conventions for the indices: $i,j,k = 1 , \dots , n$ are the fiber indices, $\lambda , \mu , \nu = 1, \dots , m$ are the manifold indices and $a,b,c = n+1 , \dots , n+ \kappa^n_{m,p}$ are the extension indices.  We also adopt the \mbox{Einstein} summation convention, i.e., there is a summation when the same index is repeated in high and low positions. However, we will write the sign $\sum$ and make explicit the values of the summation indices when it is necessary. Let $\eta=(\eta^1 , \dots , \eta^m)$ be a moving coframe on $\mathcal{M}$. Let $E=(E_{1}, \dots , E_{n})$ be an $g$-orthonormal  moving frame of $\mathbb{V}$. The covariantly closed $\mathbb{V}$-valued differential $p$-form $\phi \in \Gamma(\wedge^p\mathrm{T}^\ast\mathcal{M}\otimes\mathbb{V}_{n})$ can be expressed as follows:
\begin{equation}
\phi=E_{i}\phi^{i}= E_{i}\psi^{i}_{\lambda_{1}, \dots , \lambda_{p}}\eta^{\lambda_{1}, \dots , \lambda_{p}}
\end{equation}
where $\psi^{i}_{\lambda_{1}, \dots , \lambda_{p}}$ are functions on $\mathcal{M}$, we assume  that  $1\leqslant \lambda_{1} < \dots < \lambda_{p} \leqslant m$ in the summation and $\eta^{\lambda_{1}, \dots , \lambda_{p}}$ means $\eta^{\lambda_{1}}\wedge\dots\wedge\eta^{\lambda_{p}}$.

\begin{Def}
Let $\phi \in \Gamma(\wedge^p\mathrm{T}^\ast\mathcal{M}\otimes \mathbb{V})$  be a $\mathbb{V}$-valued differential $p$-form on $\mathcal{M}$. The generalized torsion of a connection relative to $\phi$ (or for short, a $\phi$-torsion) on a vector bundle over $\mathcal{M}$   is a $\mathbb{V}$-valued differential  $(p+1)$-form $\Theta=(\Theta^{i}):=\mathrm{d}_{\nabla}\phi$, i.e., in a local frame
\begin{equation}
\Theta = E_{i}\Theta^{i}:=E_{i}(\mathrm{d}\phi^{i}+ \eta^{i}_{j}\wedge\phi^{j}) \end{equation}
where $(\eta^{i}_{j})$ is the connection 1-form of $\nabla$ which is an $ \mathfrak{o}(n)$-valued differential 1-form (since  $\nabla$ is compatible with the metric bundle).
\end{Def}

Hence, the condition of being covariantly closed $\displaystyle{\mathrm{d}_{\nabla}\phi = 0}$ is equivalent to the fact that,  $\mathrm{d}\phi^{i}+\eta^{i}_{j}\wedge\phi^j = 0$ for all  $i = 1, \dots , n$. From the above definition, the connection $\nabla$ is then said to be $\phi$-torsion free. We also notice  that the generalized torsion defined above reduces to the standard torsion in the tangent bundle case when $\phi=E_{i}\psi^{i}_{\lambda}\eta^\lambda = E_{i}\eta^{i}$ (since $\phi=\mathrm{Id}_{\mathrm{T}\mathcal{M}}$, the functions  $\psi^{i}_{\lambda} = \delta^{i}_{\lambda}$ are the Kronecker tensors), and  the connection is Levi-Civita.\\

Let us assume that the problem has a solution. We consider the flat connection 1-form $\omega$ on the Stiefel space $SO(n+\kappa^n_{m,p})/SO(\kappa^n_{m, p})$, the $n$-adapted frames of $\mathbb{R}^{(n+\kappa^n_{m,p})}$, i.e., the set of orthonormal families of $n$ vectors $(e_{1}, \dots , e_{n})$ of $\mathbb{R}^{(n+\kappa^n_{m,p})}$ which can be completed by orthonormal $\kappa^n_{m,p}$ vectors $(e_{n+1}, \dots,e_{n+\kappa^n_{m,p}})$ to obtain an orthonormal set of $(n+\kappa^n_{m,p})$ vectors. Since we work locally, we will assume without loss of generality that we are given a cross-section $(e_{n+1},\dots ,e_{n+\kappa^n_{m,p}} )$ of the bundle fibration $SO(n+\kappa^n_{m,p})\longrightarrow SO(n+\kappa^n_{m,p}/SO(\kappa^n_{m,p})$. The flat standard 1-form of the connection $\omega$ is defined as follows: $\omega^{i}_{j} = \langle e_{i},\mathrm{d}e_{j}\rangle$ and $\omega^{a}_{i} = \langle e_{a}, \mathrm{d}e_{i}\rangle$, where $\langle,\rangle$ is the standard inner product on $\mathbb{R}^{n+\kappa^n_{m,p}}$. Notice that $\omega$ satisfies the Cartan structure equations. Suppose now that such an isometric embedding exists, then,  if $e_{i}=\Psi(E_{i})$, the condition $\mathrm{d}\Psi(\phi)=0$ yields to
\begin{equation}
e_{i}(\mathrm{d}\phi^{i} + \omega^{i}_{j}\wedge\phi^j) + e_{a}(\omega^{a}_{i}\wedge\phi^{i})=0,
\end{equation}
a condition which is satisfied  iff
\begin{equation}
\eta^{i}_{j} = \Psi^{\ast}(\omega^{i}_{j}) \quad \text{ and }\quad \Psi^{\ast}(\omega^{a}_{i})\wedge\phi^{i} = 0.
\end{equation}

The problem then turns  to  finding moving frames $(e_{1}, \dots ,  e_{n},$\linebreak $e_{n+1},
\dots ,e_{n+\kappa^n_{m,p}})$ such that  there exist  $m$-dimensional integral manifolds of the exterior ideal generated by the naive exterior differential system $\{ \omega^{i}_{j} - \eta^{i}_{j} , \omega^{a}_{i}\wedge\phi^{i}\}$  on the product manifold
\begin{equation}
\mathbf{\Sigma}^n_{m,p} = \mathcal{M}\times \frac{SO(n+\kappa^n_{m,p})}{SO(\kappa^n_{m,p})}.
\end{equation}

Strictly speaking, the differential forms live in different spaces. Indeed, one should consider the projections $\pi_{\mathcal{M}}$ and $\pi_{St}$ of $\mathbf{\Sigma}^n_{m,p}$ on $\mathcal{M}$ and the Stiefel space and consider the ideal  on $\mathbf{\Sigma}^n_{m,p}$ generated by   $\pi^\ast_{\mathcal{M}}(\eta^{i}_{j}) - \pi^\ast_{St}(\omega^{i}_{j})$  and $\pi^\ast_{St}(\omega^{a}_{i})\wedge\pi^\ast_{\mathcal{M}}(\phi^{i})$. It seems reasonable however to simply write $\{\omega^{i}_{j}-\eta^{i}_{j} , \omega^{a}_{i}\wedge\phi^{i}\}$.\\

To find integral manifolds of the naive EDS, we  would need to check  that the exterior ideal is closed under the differentiation. However, this  turns out not to be the case. The idea is then to add to the naive EDS the exterior differential of the forms that generate it and therefore, we obtain a closed one. \\

Let us notice that some objects which we are dealing with in the following,  have a geometric meaning in the tangent bundle case with a standard 1-form ($\phi = \mathrm{Id}_{\mathrm{T}\mathcal{M}}$) but not in an arbitrary vector bundle case as we noticed earlier with the notion of torsion of a connection. That leads us to notions in a generalized sense in such a way that we recover the standard notions in the tangent bundle case. First of all, the Cartan lemma,  which in the isometric embedding problem implies the symmetry of the second fundamental form,  does not hold. Consequently, we can not assure nor assume that the coefficients of the second fundamental form are symmetric as in the isometric embedding problem. In fact, we will show that these conditions should be replaced by  \textit{generalized Cartan identities} that express how coefficients of the second fundamental form are related to each other, and of course, we recover the usual symmetry in the tangent bundle case. Another difficulty is the analogue of the  Bianchi identity of the curvature tensor. We will define  \textit{generalized Bianchi identities} relative to the covariantly closed vector valued differential $p$-form and a \textit{generalized curvature tensor space} which correspond, in the tangent bundle case, to the usual Bianchi identities and the Riemann curvature tensor space respectively. Finally, besides the \textit{generalized Cartan identities} and \textit{generalized curvature tensor space}, we will make use of a \textit{generalized Gauss map}. \\

The key to  the proof of Theorem \ref{CLGIEVBThm} is Lemma \ref{LemFond} for two main reasons: on one hand, it assures the existence of suitable coefficients of the second fundamental form that satisfy the \textit{generalized Cartan identities} and the \textit{generalized Gauss equation}, properties that simplify the computation of the Cartan characters. On the other hand, the lemma gives the minimal required \textit{embedding codimension} $\kappa^n_{m,m-1}$ that ensures the desired embedding. Using Lemma \ref{LemFond}, we give  another proof of Theorem \ref{CLGIEVBThm} by an explicit construction of an ordinary integral flag. When the existence of integral manifold is established, we just need to project it on $\mathcal{M}\times\mathbb{R}^{n+\kappa^n_{m,p}}$.

This naive EDS is not closed. Indeed, the  generalized torsion-free of the connection implies that $\mathrm{d}(\omega^{a}_{i}\wedge\phi^{i}) \equiv 0 $ modulo the naive EDS, but the second Cartan structure equation yields to  $\displaystyle{\mathrm{d}(\omega^{i}_{j} - \eta^{i}_{j}) \equiv \omega^{i}_{a}\wedge\omega^{a}_{j} + \Omega^{i}_{j}}$ modulo the naive EDS, where $\Omega = (\Omega^{i}_{j})$ is the curvature 2-form of the connection. Consequently, the exterior ideal that we now consider on $\mathbf{\Sigma}^n_{m,p}$ is  
\begin{equation}
\mathcal{I}^n_{m,p} = \{ \omega^{i}_{j} - \eta^{i}_{j} ,  \omega^{i}_{a}\wedge\omega^{a}_{j} + \Omega^{i}_{j} ,  \omega^{a}_{i}\wedge\phi^{i} \}_{\text{alg}}
\end{equation}
The curvature 2-form of the connection is an $\mathfrak{o}(n)$-valued two form and is related to the connection 1-form $(\eta^{i}_{j})$ by the Cartan structure equation:
\begin{equation}\label{CartanStructureEq}
\Omega^{i}_{j}= \mathrm{d}\eta^{i}_{j} + \eta^{i}_{k}\wedge\eta^{k}_{j}
\end{equation}
A first covariant derivative of $\phi$ has led to the generalized torsion. A second covariant derivative of $\phi$ gives rise to  \textit{generalized Bianchi identities}\footnote{In the tangent bundle case and $\phi = E_{i}\eta^{i}$, we recover the standard Bianchi identities of the Riemann curvature tensor, i.e, $\mathcal{R}^{i}_{jkl} =\mathcal{R}^{k}_{lij}$ and  $\mathcal{R}^{i}_{jkl} + \mathcal{R}^{i}_{ljk}  + \mathcal{R}^{i}_{klj}  = 0 $.} as follows:
\begin{equation}
\mathrm{d}_{\nabla}^2(\phi) = 0 \Longleftrightarrow \Omega^{i}_{j}\wedge\phi^{j} = 0 \text{ for all } i = 1, \dots , n.
\end{equation}
The conditions  $\Omega^{i}_{j}\wedge \phi^{j}=0$ for all $i = 1, \dots , n$ are called \textit{generalized Bianchi identities}. We then define  a generalized curvature tensor space $\mathcal{K}^n_{m,p}$ as the space of curvature tensor satisfying the \textit{generalized Bianchi identities}:
\begin{equation}
\mathcal{K}^n_{m,p}=\{ (\mathcal{R}^{i}_{j;\lambda \mu}) \in \wedge^2(\mathbb{R}^n)\otimes\wedge^2(\mathbb{R}^m) | \, \Omega^{i}_{j}\wedge\phi^{j} = 0, \forall i=1,\dots, n\}
\end{equation}
where $\Omega^{i}_{j} = \frac{1}{2}\mathcal{R}^{i}_{j;\lambda\mu}\eta^\lambda\wedge\eta^\mu$. \\

In the tangent bundle case,  $\mathcal{K}^n_{n,1}$ is the Riemann curvature tensor space which is of dimension $m^2(m^2-1)/12$. 
\\


All the data are analytic, we can then apply the Cartan--K\"{a}hler theory if we are able to check the involution of  the exterior differential system by constructing an $m$-integral flag: If the exterior ideal $\mathcal{I}^n_{m,p}$ satisfies the Cartan test, the flag is then ordinary and by the Cartan--K\"{a}hler theorem, there exist integral manifolds of $\mathcal{I}^n_{m,p}$. To be able to project the product manifold $\mathbf{\Sigma}^n_{m,p}$ on $\mathcal{M}$, we  also need to show the existence of $m$-dimensional integral manifolds on which the volume form on $\eta^{1,\dots , m}$ on  $\mathcal{M}$ does not vanish. 
\\ 
 
Let us express the  $1$-forms $\omega^{a}_{i}$ in  the coframe $(\eta^1, \dots , \eta^m)$ in order to later make  the computation of Cartan characters easier. Let $\mathcal{W}^n_{m,p}$ be an $\kappa^n_{m,p}$-dimensional Euclidean space. One can think of it as a normal space.  We then write $\omega^{a}_{i} = H^a_{i\lambda}\eta^{\lambda}$ where $H^{a}_{i\lambda}\in \mathcal{W}^n_{m,m-1}\otimes \mathbb{R}^n\otimes\mathbb{R}^m$ and define the forms $\pi^{a}_{i} = \omega^{a}_{i}  - H^a_{i\lambda}\eta^{\lambda}$. We can also consider $H_{i\lambda}= (H^{a}_{i\lambda})$ as a vector of $\mathcal{W}^{n}_{m,p}$. The forms that generate algebraically $\mathcal{I}^n_{m,p}$ are then expressed as follows:
\begin{equation}\label{GaussGenExpression}
\begin{split}
\sum_{a}\omega^{a}_{i}\wedge\omega^{a}_{j}-\Omega^{i}_{j}& = \sum_{a}\pi^{a}_{i}\wedge\pi^{a}_{i} +\sum_{a} (H^{a}_{j\lambda} \pi^{a}_{i}- H^{a}_{i\lambda}\pi^{a}_{j})\wedge\eta^\lambda \\
& \hspace*{3ex}+\underbrace{\frac{1}{2}(H_{i\lambda}.H_{j\mu} - H_{i\mu}.H_{j\lambda} - \mathcal{R}^{i}_{j;\lambda\mu})}_{*}\eta^\lambda\wedge\eta^\mu
\end{split}
\end{equation}
and
\begin{equation}\label{CartanGenExpression}
\omega^{a}_{i}\wedge\phi^{i}= \psi^{i}_{\lambda_{1}\dots\lambda_{p}}\pi^{a}_{i}\wedge\eta^{\lambda_{1}\dots \lambda_{p}} + \hspace{-1.2cm}\sum_{\tiny{\begin{array}{c}\lambda = 1,\dots , m \\1\leqslant \mu_{1}<\dots <\mu_{p}\leqslant m\end{array}}}\hspace{-1.2cm}\overbrace{H^{a}_{i\lambda}\psi^{i}_{\mu_{1}, \dots , \mu_{p}}}^{**}\eta^{\lambda\mu_{1}\dots\mu_{p}}
\end{equation}

These new expressions of the forms in terms of vectors $H$ and the differential $1$-form $\pi$ will help us  to compute the Cartan characters of an $m$-integral flag. To simplify these calculations, we will choose $H^{a}_{i\lambda}$, which are the coefficients of the second fundamental form, so that the quantities marked with $(*)$ and $(**)$ in the equation (\ref{GaussGenExpression}) and (\ref{CartanGenExpression}) vanish and hence:
\begin{eqnarray}
\label{GenGaussEq}H_{i\lambda}.H_{j\mu} - H_{i\mu}.H_{j\lambda}= \mathcal{R}^{i}_{j;\lambda\mu} \quad \text{generalized Gauss equation }\\
\label{GenCartanId}\sum_{\tiny{\begin{array}{c}\lambda = 1,\dots , m \\1\leqslant \mu_{1}<\dots <\mu_{p}\leqslant m\end{array}}}\hspace{-1.2cm}H^{a}_{i\lambda}\psi^{i}_{\mu_{1}, \dots , \mu_{p}}\eta^{\lambda\mu_{1}\dots\mu_{p}} =0\quad \text{generalized Cartan identities }
\end{eqnarray}

As mentioned previously in the introduction, the system of equations (\ref{GenCartanId}) is said to be \textit{generalized Cartan identities} because it gives us  relations between the coefficients of the second fundamental form which are not necessarily the usual symmetry given by the Cartan lemma. These properties of the coefficients and the fact that  the curvature tensor $(\mathcal{R}^{i}_{j;\lambda\mu})$ satisfies generalized Bianchi identities lead us to name the equation (\ref{GenGaussEq}) as the \textit{generalized Gauss equation}. 

We now  define a \textit{generalized Gauss map} $\mathcal{G}^n_{m,p}:\mathcal{W}^n_{m,p}\otimes \mathbb{R}^n\otimes\mathbb{R}^m \longrightarrow \mathcal{K}^n_{m,p}$ defined for $H^{a}_{i\lambda}\in \mathcal{W}^n_{m,p}\otimes \mathbb{R}^n\otimes\mathbb{R}^m $ by 
\begin{equation}
\Big(\mathcal{G}^n_{m,p}(H)\Big)^{i}_{j;\lambda\mu}= \sum_{a}(H^{a}_{i\lambda}H^{a}_{j\mu} - H^{a}_{i\mu}H^{a}_{j\lambda})
\end{equation}

Let us specialize in the conservation laws case, i.e., when $p=m-1$. We adopt the following notations: $\Lambda=(1,2,\dots, m)$ and $\Lambda\smallsetminus k = (1,\dots , k-1 ,k+1 ,\dots , m)$. We have thus  $\eta^{\Lambda} = \eta^1\wedge\dots\wedge\eta^m$ and $\eta^{\Lambda \smallsetminus k}= \eta^1\wedge\dots\wedge\eta^{k-1}\wedge\eta^{k+1}\dots\wedge\eta^m$. Let us construct an ordinary $m$-dimensional integral element of the exterior ideal $\mathcal{I}^n_{m,m-1}$ on $\mathbf{\Sigma}^n_{m,m-1}$. Generalized Bianchi identities are trivial in this case and so $\displaystyle{\text{dim}\,\mathcal{K}^n_{m,m-1} = \frac{n(n-1)}{2}\frac{m(m-1)}{2}}$.

The generalized Gauss equation is $H_{i\lambda}.H_{j\mu} - H_{i\mu}.H_{j\lambda} = \mathcal{R}^{i}_{j;\lambda\mu}$, where $H_{i\lambda}$ is viewed as a vector of the $\kappa^n_{m,m-1}$-Euclidean space $\mathcal{W}^{n}_{m,m-1}$. Generalized Cartan identities are
\begin{equation}
\sum_{\lambda = 1,\dots , m}(-1)^{\lambda+1} H^{a}_{i\lambda}\psi^{i}_{\Lambda \smallsetminus \lambda} = 0 \qquad \text{for all } a 
\end{equation}

The following lemma was proved in \cite{LCLGIEVB-Art} and is the key to the proof of  Theorem \ref{CLGIEVBThm} .

\begin{Lem}\label{LemFond}
Let $\kappa^n_{m,m-1}\geqslant (m-1)(n-1)$. Let $\mathcal{H}^n_{m,m-1} \subset \mathcal{W}^n_{m,m-1}\otimes \mathbb{R}^n\otimes\mathbb{R}^m $ be the open set consisting of those elements $H=(H^a_{i\lambda})$ so that the vectors $\{ H_{i\lambda} | i=1,\dots , n-1  \text{ and } \lambda = 1, \dots , m-1\}$ are linearly independents as elements of $\mathcal{W}^n_{m,m-1}$ and satisfy  generalized Cartan identities. Then 
$ \mathcal{G}^n_{m,m-1} : \mathcal{H}^n_{m,m-1} \longrightarrow \mathcal{K}^n_{m,m-1}$ is a surjective submersion.
\end{Lem}

Let $\mathcal{Z}^n_{m,m-1}=\{ (x, \Upsilon, H) \in \mathbf{\Sigma}^n_{m,m-1}\times \mathcal{H}^n_{m,m-1} \, |\, \mathcal{G}^n_{m,m-1}(H)=\mathcal{R}(x) \}$. We conclude from Lemma \ref{LemFond} that $\mathcal{Z}^n_{m,m-1}$ is a submanifold\footnote{$\mathcal{Z}^n_{m,m-1}$ is the fiber of $\mathcal{R}$ by a submersion. The surjectivity of $\mathcal{G}^n_{m,m-1}$ assures the non-emptiness.} and hence, 
\begin{equation}
\dim\, \mathcal{Z}^n_{m,m-1} = \dim \, \mathbf{\Sigma}^n_{m,m-1} + \dim \, \mathcal{H}^n_{m,m-1} 
\end{equation}
where
\begin{eqnarray}
\dim \, \mathbf{\Sigma}^n_{m,m-1} = m+\frac{n(n-1)}{2}+n\kappa^n_{m,m-1}\\
\dim\, \mathcal{H}^n_{m,m-1} =(nm-1)\kappa^n_{m,m-1} - \frac{n(n-1)m(m-1)}{4} 
\end{eqnarray}

We define the map $\Phi^n_{m,m-1}:\mathcal{Z}^n_{m,m-1}\longrightarrow \mathcal{V}_{m}(\mathcal{I}^n_{m,m-1} , \eta^\Lambda)$ which associates $(x,\Upsilon , H)\in \mathcal{Z}^n_{m,m-1}$ with  the $m$-plan on which the differential forms that generate algebraically $\mathcal{I}^n_{m,m-1}$ vanish and the volume form $\eta^\Lambda$ on $\mathcal{M}$ does not vanish. $\Phi^n_{m,m-1}$ is then an embedding and hence $\dim\, \Phi(\mathcal{Z}^n_{m,m-1})= \dim\, \mathcal{Z}^n_{m,m-1}$. In what follows, we prove that in fact $\Phi(\mathcal{Z}^n_{m,m-1})$ contains only ordinary $m$-integral elements of $\mathcal{I}^n_{m,m-1}$. Since the coefficients $H^a_{i\lambda}$ satisfy the generalized Gauss equation and generalized Cartan identities, the differential  forms that generate the exterior ideal $\mathcal{I}^n_{m,m-1}$ are as follows:
\begin{eqnarray}
\omega^{i}_{a}\wedge\omega^{a}_{j}+\Omega^{i}_{j} = \sum_{a}\pi^{a}_{i}\wedge\pi^{a}_{i} +\sum_{a} (H^{a}_{j\lambda} \pi^{a}_{i}- H^{a}_{i\lambda}\pi^{a}_{j})\eta^\lambda \\
\omega^{a}_{i}\wedge\phi^{i}= \psi^{i}_{\lambda_{1}\dots\lambda_{p}}\pi^{a}_{i}\wedge\eta^{\lambda_{1}\dots \lambda_{p}}
\end{eqnarray}

The final step is then computing the Cartan characters and checking by the Cartan's test that $\Phi(\mathcal{Z}^n_{m,m-1})$ contains only ordinary $m$-integral flags. The Cartan--K\"{a}hler theorem then assures the existence of an $m$-integral manifold on which $\eta^\Lambda$ does not  vanish since the exterior ideal is in involution. We finally project the integral manifold on $\mathcal{M}\times \mathbb{R}^{n+\kappa}$. Let us notice that the requirement of the non-vanishing of  the volume form  $\eta^\Lambda$ on the integral manifold yields to project the integral manifold on $\mathcal{M}$ and also to view it as a graph of a function $f$ defined on $\mathcal{M}$ with values in the space of $n$-adapted orthonormal frames of $\mathbb{R}^{n+\kappa}$. In  the isometric embedding problem, the composition of $f$ with the projection of the frames on the Euclidean space is by construction the  isometric embedding map.

\section[Conservation laws ...]{Conservation laws for covariant divergence-free energy-momentum tensors}

We present here an application for Theorem \ref{CLGIEVBThm} to covariant divergence free energy-momentum tensors. 
\begin{Prop}
Let $(\mathcal{M}^{m}, g)$ be a real analytic $m$-dimensional Riemannian manifold, $\nabla$ be the Levi-Civita connection and $T$ be a contravariant 2-tensor with a vanishing covariant divergence. Then  there exists a conservation law of $T$ on $\mathcal{M}\times\mathbb{R}^{m+(m-1)^2}$.
\end{Prop}

\begin{proof}
Let us consider a contravariant 2-tensor  $T\in \Gamma(\mathrm{T}\mathcal{M}\otimes \mathrm{T}\mathcal{M})$ which  is expressed in a chart by $T=T^{ij}\xi_{i}\otimes\xi_{j}$, where $(\xi_{1} , \dots , \xi_{m})$ is the dual basis of an orthonormal moving coframe $(\eta^1 , \dots , \eta^m)$. The volume form  is denoted by $\eta^{I} = \eta^1\wedge\dots \wedge\eta^m$. Using the interior product, we can associate  any contravariant 2-tensor  $T$ with a $\mathrm{T}\mathcal{M}$-valued $m$-differential form $\tau$ defined as follows:
\begin{equation*}
\begin{split}
\Gamma (\mathrm{T}\mathcal{M} \otimes \mathrm{T}\mathcal{M}) &\longrightarrow \Gamma (\mathrm{T}\mathcal{M}\otimes \wedge^{(m-1)}\mathrm{T}^\ast \mathcal{M}).\\
T=T^{ij}\xi_{i}\otimes\xi_{j} & \longmapsto \tau = \xi_{i}\otimes \tau^{i}
\end{split}
\end{equation*}
where $\displaystyle{ \tau^{i} = T^{ij}(\xi_{j} \lrcorner \eta^{I})}$

The tangent space $\mathrm{T}\mathcal{M}$ is endowed with the Levi-Civita connection $\nabla$. Let us compute the covariant derivative of $\tau$. 
\begin{equation}
\mathrm{d}_{\nabla}\tau = \xi_{i}\otimes (\mathrm{d}\tau^{i} + \eta^{i}_{j}\wedge\tau^j)
\end{equation}
 On one hand, using the first Cartan equation that expresses the vanishing of the torsion of the Levi-Civita connection and the expression of the Christoffel symbols\footnote{The Christoffel symbols are the functions $\Gamma$ defined by: $\eta^{i}_{j} = \Gamma^{i}_{kj}\eta^k.$} in term of the connection 1-form, we obtain
\begin{equation}
\begin{split}
\mathrm{d}\tau^i &= \mathrm{d}\Big(T^{ij}(\xi_{j}\lrcorner \eta^{I})\Big) = \mathrm{d}(T^{ij})\wedge(\xi_{j}\lrcorner \eta^{I})+ T^{ij}\mathrm{d}(\xi_{j}\lrcorner \eta^{I})\\
&= \Big(\xi_{j}(T^{ij}) + T^{ij}\Gamma^{k}_{kj}\Big)\eta^{I} 
\end{split}
\end{equation}
and
\begin{equation}
\eta^{i}_{j}\wedge\tau^{j} = \eta^{i}_{j}\wedge T_{jk}(\xi_{k}\lrcorner\eta^{I})= \Big( T^{jk}\Gamma^{i}_{kj}\Big)\eta^{I}
\end{equation}
consequently
\begin{equation}
\mathrm{d}_{\nabla}\tau = \xi_{i}\otimes\Big[\Big(\xi_{j}(T^{ij}) + T^{ij}\Gamma^{k}_{kj}+ T^{jk}\Gamma^{i}_{kj}\Big)\eta^{I} \Big]
\end{equation}
On the other hand, a straightforward computation of the divergence of the contravariant 2-tensor leads us to
\begin{equation}
\nabla_{j}T^{ij} = \xi_{j}(T^{ij}) + T^{ij}\Gamma^{k}_{kj}+ T^{jk}\Gamma^{i}_{kj}
\end{equation}
 We conclude then that 
 \begin{equation}
\mathrm{d}_{\nabla}\tau = 0 \Leftrightarrow \nabla_{j}T^{ij}=0
\end{equation}

Hence, for an $m$-dimensional Riemannian manifold $\mathcal{M}$, the main result of this article assures the existence of an isometric embedding $\Psi: \mathrm{T}\mathcal{M} \longrightarrow \mathcal{M}\times \mathbb{R}^{m+(m-1)^2}$ such that $\mathrm{d}(\Psi (\tau))=0$ is a conservation law for a  covariant divergence-free energy-momentum tensor. 
\end{proof}

For instance, if $\mathrm{dim}\mathcal{M} = 4$, then $\Psi(\tau)$ is a closed differential $3$-form on $\mathcal{M}$ with values in $\mathbb{R}^{13}$.


\end{document}